\documentclass{amsart}
\usepackage{amsfonts,amssymb, amscd}
\usepackage{amsmath}
\usepackage{lmodern}
\usepackage{amsthm}
\usepackage{qsymbols}
\usepackage{chngcntr}
\usepackage{mathrsfs}
\usepackage[noadjust]{cite}

\newtheorem{thm}{Theorem}
\newtheorem{lem}[thm]{Lemma}
\newtheorem{prop}[thm]{Proposition}
\newtheorem{cor}[thm]{Corollary}
\newtheorem{pb}[thm]{Problem}
\newtheorem{conj}[thm]{Conjecture}
\theoremstyle{definition}
\newtheorem{rem}[thm]{Remark}
\newcommand{\IR}{\mathbb{R}}
\newcommand{\IC}{\mathbb{C}}
\newcommand{\cS}{\mathcal{S}}
\renewcommand{\L}{\mathrm{L}}

\renewcommand{\H}{\mathrm{H}}
\newcommand{\BMO}{\mathrm{BMO}}
\newcommand{\E}{\mathsf{E}}
\newcommand{\e}{\mathrm{e}}
\renewcommand{\d}{\; \mathrm{d}}

\renewcommand{\i}{\mathrm{i}}
\newcommand{\eps}{\varepsilon}

\renewcommand\Re{\operatorname{Re}}
\renewcommand\Im{\operatorname{Im}}
\newcommand{\Lop}{\mathcal{L}}
\newcommand{\cl}[1]{\overline{#1}}
\renewcommand\div{\operatorname{div}}
\newcommand{\dhalf}{D_t^{1/2}} 
\newcommand{\HT}{H_t}
\newcommand{\sgn}{\operatorname{sgn}}
\newcommand{\essup}{\operatorname{essup}}
\renewcommand{\a}{\mathfrak{a}}
\newcommand{\A}{\mathfrak{A}}
\newcommand{\fe}{\mathfrak{e}}
\newcommand{\fL}{\mathfrak{L}}
\def\Xint#1{\mathchoice
{\XXint\displaystyle\textstyle{#1}}%
{\XXint\textstyle\scriptstyle{#1}}%
{\XXint\scriptstyle\scriptscriptstyle{#1}}%
{\XXint\scriptscriptstyle%
\scriptscriptstyle{#1}}%
\!\int}
\def\XXint#1#2#3{{\setbox0=\hbox{$#1{#2#3}{%
\int}$ }
\vcenter{\hbox{$#2#3$ }}\kern-.6\wd0}}
\def\barint{\,\Xint-} 
\title[Non-autonomous maximal regularity for elliptic operators]{On non-autonomous maximal regularity for elliptic operators in divergence form}
\author{Pascal Auscher}
\author{Moritz Egert}
\address{Laboratoire de Math\'{e}matiques d'Orsay, Univ. Paris-Sud, CNRS, Universit\'{e} Paris-Saclay, 91405 Orsay, France}
\email{pascal.auscher@math.u-psud.fr, moritz.egert@math.u-psud.fr}
\thanks{The authors were partially supported by the ANR project ``Harmonic Analysis at its Boundaries'', ANR-12-BS01-0013. M.E.\ was supported by a public grant as part of the FMJH}
\subjclass[2010]{Primary: 35K15. Secondary: 47A07, 26A33.} 
\date{\today}
\dedicatory{}
\keywords{non-autonomous forms, maximal regularity, elliptic differential operators, fractional derivatives, commutator estimates}
\begin{document}
\begin{abstract}
We consider the Cauchy problem for non-autonomous forms inducing elliptic operators in divergence form with Dirichlet, Neumann, or mixed boundary conditions on an open subset $\Omega \subseteq \IR^n$. We obtain maximal regularity in $\L^2(\Omega)$ if the coefficients are bounded, uniformly elliptic, and satisfy a scale invariant bound on their fractional time-derivative of order one-half. Previous results even for such forms required control on a time-derivative of order larger than one-half.
\end{abstract}
\maketitle
\section{Introduction}
\label{Sec: Introduction}

\noindent Let $V$ be a complex Hilbert space, $V^*$ be the anti-dual space of conjugate-linear functionals on $V$, and $H$ be a second complex Hilbert space in which $V$ densely embeds. Suppose that $\a: [0,T] \times V \times V \to \IC$ is a strongly measurable, bounded, quasi-coercive, non-autonomous form: Each $\a(t, \cdot, \cdot)$ is a sesquilinear form on $V$ and there exist constants $\Lambda, \lambda, \eta > 0$ such that 
\begin{align}
\label{Eq: Quasi coercivity}
 |\a(t,v,w)| \leq \Lambda \|v\|_V \|w\|_V \qquad \text{and} \qquad  \Re \a(t,v,v) \geq \lambda \|v\|_V^2 - \eta \|v\|_H^2
\end{align}
hold for all $v,w \in V$ and $t \in [0,T]$. Each form $\a(t,\cdot, \cdot)$ induces a bounded operator $\A(t) \in \Lop(V, V^*)$ via $\langle \A(t)v , w \rangle = \a(t, v, w)$. A classical result due to J.L.~Lions states that for each $f \in \L^2(0,T; V^*)$ the non-autonomous Cauchy problem
\begin{align}
\label{Eq: Cauchy}
 u'(t) + \A(t)u(t) = f(t), \qquad u(0)=0
\end{align}
has a unique solution $u \in \H^1(0,T; V^*) \cap \L^2(0,T; V)$, see \cite[p.~513]{Dautray-Lions}. This is usually rephrased as saying that $\a$ admits maximal regularity in $V^*$. We remark that the original argument needs that $H$ is separable but this assumption is not necessary due to a new proof of Dier and Zacher \cite[Thm.~6.1]{Dier-Zacher}. A famous problem, first posed explicitly by J.L.~Lions in 1961 (see \cite[p.~68]{Lions-Problem}), concerns maximal regularity in the smaller space $H$:

\begin{pb}
\label{Pb: Lions problem}
Is it true that for every $f \in \L^2(0,T;H)$ the unique solution $u$ of \eqref{Eq: Cauchy} belongs to the space $\H^1(0,T; H)$?
\end{pb}

In the autonomous case $\A(t) = \A(0)$, de Simon proved in 1964 that maximal regularity in $H$ holds true if and only if the part of $-\A(0)$ in $H$ generates a holomorphic $C_0$-semigroup~\cite{DeSimon}. Recent progress in the non-autonomous case has thrust the $\alpha$-H\"older continuity
\begin{align}
\label{Eq: Holder condition}
 \|\A(t)-\A(s)\|_{V \to V^*} \leq C|t-s|^\alpha \qquad (t,s \in [0,T])
\end{align}
of $\A$ into the spotlight: On the one hand Ouhabaz and Spina answered Problem~\ref{Pb: Lions problem} in the affirmative if $\A$ is H\"older continuous of exponent $\alpha > \frac{1}{2}$, see \cite[Thm.~3.3]{Ouhabaz-Spina}. Astonishingly, Fackler on the other hand was able to construct a symmetric non-autonomous form that is $\alpha$-H\"older continuous for every $\alpha < \frac{1}{2}$ but still fails maximal regularity in $H$, see \cite[Thm.~5.1]{Fackler-QE}. Dier and Zacher \cite{Dier-Zacher} replaced the classical H\"older assumption by its square-integrated version
\begin{align}
\label{Eq: Fractional regularity}
\int_0^T \int_0^T \frac{\|\A(t)-\A(s)\|_{V \to V^*}^2}{|t-s|^{2\alpha}} \; \frac{\mathrm{d} s \; \mathrm{d} t}{|t-s|} < \infty,
\end{align}
usually referred to as fractional $\L^2$-Sobolev regularity of order $\alpha$. However, in order to prove maximal regularity in $H$ they had again to assume $\alpha > \frac{1}{2}$. 

All these results left open the borderline case of $\frac{1}{2}$-regularity. In connection with his counterexample, Fackler \cite{Fackler-QE} raised the question whether maximal regularity in the case $\alpha = \frac{1}{2}$ holds if $\a$ induces elliptic operators in divergence form. In this note we shall provide a first positive answer to this question: If $\a$ induces elliptic operators in divergence form with real symmetric coefficients, then a sufficient condition for maximal regularity in $H$ is
\begin{align}
\label{Eq: First BMO condition}
\sup_{I \subseteq [0,T]} \frac{1}{\ell(I)} \int_I \int_I \frac{\|\A(t)-\A(s)\|_{V \to V^*}^2}{|t-s|} \; \frac{\mathrm{d} s \; \mathrm{d} t}{|t-s|} < \infty,
\end{align}
where $I$ is an interval and $\ell(I)$ its length. This is the scale invariant version of Dier and Zacher's condition \eqref{Eq: Fractional regularity} in the borderline case $\alpha = \frac{1}{2}$, but \eqref{Eq: Fractional regularity} and \eqref{Eq: First BMO condition} do not compare. Moreover, we obtain maximal regularity in $H$ for complex and possibly non-symmetric coefficients if we impose a condition akin to \eqref{Eq: First BMO condition} directly on the coefficients. In order to state our result more precisely, we need to recall some standard notation.

\subsection{Notation and precise statement of the main result}
\label{Subsec: Notation}

The John-Nirenberg space $\BMO(\IR)$ is defined as the space of locally integrable functions $f$ modulo constants that have bounded mean oscillation
\begin{align*}
 \|f\|_{\BMO(\IR)} := \sup_{I} \barint_I |f(t) - f_I| \d t < \infty,
\end{align*}
where the supremum runs over all bounded intervals $I \subseteq \IR$ and $f_I$ denotes the average on $I$. For $\alpha \in (0,1)$ the fractional $t$-derivative $D_t^\alpha$ is defined on the space $\cS'(\IR)/ \mathcal{P}$ of tempered distributions modulo polynomials by the Fourier symbol $|\tau|^\alpha$. 

By a non-autonomous form $\a$ inducing elliptic operators in divergence form with either Dirichlet, Neumann, or mixed boundary conditions, we always mean the following special setup: The Hilbert space $H = \L^2(\Omega)$, where $\Omega \subseteq \IR^n$ is a non-empty open set, the Hilbert space $V$ is a closed subspace of the first-order Sobolev space $\H^1(\Omega)$ that contains $\H_0^1(\Omega)$, the closure of the test functions in $\H^1(\Omega)$, and 
\begin{align}
\label{Eq: Differential form}
 \a(t,v,w) = \int_\Omega A(t,x) \nabla v(x) \cdot \cl{\nabla w(x)} \d x,
\end{align}
where $A: [0,T] \times \Omega \to \IC^{n \times n}$ is a bounded and measurable function for which there exists $\Lambda, \lambda > 0$ such that
\begin{align}
\label{Eq: ellipticity}
\lambda |\xi|^2 \leq \Re (A(t,x)\xi \cdot \cl{\xi}) \qquad \text{and} \qquad |A(t,x)\xi \cdot \zeta| \leq \Lambda |\xi| |\zeta| 
\end{align}
for all $t \in [0,T]$, a.e.\ $x \in \Omega$, and all $\xi, \zeta \in \IC^n$. Note that such $\a$ satisfies \eqref{Eq: Quasi coercivity} with $\eta = \lambda$. The non-autonomous form introduced above induces the divergence form operators $\A(t) = -\div A(t, \cdot) \nabla \in \Lop(V, V^*)$. The boundary conditions on $\partial \Omega$ are encoded in $V$ by a formal integration by parts only if one restricts to the part of $\A(t)$ in $H$. In fact, this is the reason why for forms inducing differential operators the notion of maximal regularity in $H$ is so much more preferable to that of maximal regularity in $V^*$. Finally, we note that for any bounded form as in \eqref{Eq: Differential form} (no need for coercivity) 
\begin{align*}
\|\A(t)\|_{V \to V^*} \lesssim \essup_{x \in \Omega} |A(t,x)|,
\end{align*}
where here and throughout $|A(t,x)|$ denotes the norm of $A(t,x)$ as an operator on the Euclidean space $\IC^n$, and that the converse estimate holds at least if $A(t,x)$ is real and symmetric for almost every $x$. The latter statement can be proved by mimicking the argument performed in \cite{Hendrik} for the lower bound of sesquilinear forms. In particular, when applied to the differences $\A(t) - \A(s)$, we see that for real symmetric coefficients the abstract H\"older condition \eqref{Eq: Holder condition} is equivalent to
\begin{align*}
 |A(t,x) - A(s,x)| \leq C |t-s|^\alpha \qquad (\text{a.e.\ $x \in \Omega$, $s, t \in [0,T]$}).
\end{align*}
Our main result can now be stated as follows.

\begin{thm}
\label{Thm: Main}
Consider a non-autonomous form $\a$ inducing elliptic divergence-form operators with either Dirichlet, Neumann, or mixed boundary conditions on an open set $\Omega \subseteq \IR^n$ as defined above. If there exists a finite $M\geq 0$ such that
\begin{align}
\label{Eq: Ass A}
\qquad \sup_{I \subseteq [0,T]} \frac{1}{\ell(I)} \int_I \int_I \frac{|A(t,x) - A(s,x)|^2}{|t-s|^{2}} \d s \d t \leq M \qquad (\text{a.e.\ $x \in \Omega$}),
\end{align}
then $\a$ has maximal regularity in $H = \L^2(\Omega)$. More precisely, given $f \in \L^2(0,T; H)$, the unique solution $u \in \H^1(0,T; V^*) \cap \L^2(0,T; V)$ of problem \eqref{Eq: Cauchy} satisfies
\begin{align*}
 \|u\|_{\H^1(0,T; H)} + \|u\|_{\H^{1/2}(0,T; V)} \leq C \|f\|_{\L^2(0,T; H)},
\end{align*}
where $C$ depends on $\lambda$, $\Lambda$, $M$, $T$, and $n$.
\end{thm}

\begin{rem}
\label{Rem: Main}
From the discussion above, we conclude that for real symmetric $A$ the condition \eqref{Eq: Ass A} is even weaker than \eqref{Eq: First BMO condition} as we have interchanged the essential supremum in $x$ with the integral sign. The additional regularity $u \in \H^{1/2}(0,T; V)$ is not expected \emph{a priori}, given the notion of maximal regularity in $H$. This seems to be a somewhat new phenomenon that first appeared in \cite{Dier-Zacher}. For background information on the vector-valued fractional Sobolev spaces the reader can refer to the appendix of \cite{Dier-Zacher}.
\end{rem}

We shall give the proof of Theorem~\ref{Thm: Main} in Section~\ref{Sec: proof} below. It relies on a reduction to the non-autonomous problem on the real line, which we shall investigate in Section~\ref{Sec: real line}. Therein, the $\L^2$-boundedness of commutators $[A(\cdot,x), \dhalf]$ for $x \in \Omega$ under our assumption on $A$ will be the crucial ingredient. Let us remark that commutator estimates have also been a central theme in Dier and Zacher's new approach to maximal regularity \cite{Dier-Zacher}. The difference is that in our special setup $A(\cdot,x)$ is valued in the finite dimensional space $\IC^{n \times n}$. Thus, we can rely on the optimal commutator bound and do not have to waste an `$\eps$ of a derivative' as is traditional in some vector-valued extensions.

\subsection{Comparison to earlier results}
\label{Subsec: Comparison}

Let us close the discussion by relating the regularity assumption in Theorem~\ref{Thm: Main} to previously introduced conditions (in the case of real symmetric coefficients). To do so rigorously, we anticipate an extension result from Lemma~\ref{Lem: extension} further below: On assuming \eqref{Eq: Ass A}, we can extend $A$ to a map $\IR \times \Omega \to \IC^{n \times n}$ in such a way that this estimate remains valid for every bounded interval $I \subseteq \IR$. Being real and symmetric is a property that is preserved under this extension. 

Taking this lemma for granted, the results of Strichartz \cite{Strichartz-BMOSobolev} on $\BMO$-Sobolev spaces yield that \eqref{Eq: Ass A} is strictly stronger than $\frac{1}{2}$-H\"older continuity of $\A$ and in fact equivalent to $A$ having an extension $A: \IR \times \Omega \to \IC^{n \times n}$ that satisfies
\begin{align*}
 \|\dhalf A(\cdot,x)\|_{\BMO(\IR)^{n \times n}} \leq C M \qquad (\text{a.e.\ $x \in \Omega$}).
\end{align*}
To see how the results in \cite{Strichartz-BMOSobolev} apply, the reader should recall that the essential supremum of $A(t)$ with respect to $x$ compares to the norm of $\A(t)$. Taking into additional account the classical embeddings of Besov spaces \cite[Sec.~V.5.2]{Stein}, we obtain that \eqref{Eq: Ass A} is strictly weaker than any of the Dini conditions
\begin{align*}
 \int_0^T \operatorname{essup}_{s \in \IR}\|\A(t+s) - \A(s)\|_{V \to V*}^q \; \frac{\mathrm{d} t}{t^{1+q/2}} < \infty,
\end{align*}
where $q \in [1,2]$. For $q=1$ this is the condition used by Ouhabaz and Haak~\cite{Haak-Ouhabaz}. We also see that any of the conditions above is implied by $\alpha$-H\"older continuity for an $\alpha > \frac{1}{2}$. Coming up with a particular example, an admissible function in product form $A(t,x) = 1+|t|^{1/2} A(x)$ will not satisfy any of the Dini conditions but it does satisfy~\eqref{Eq: Ass A} since $\dhalf |t|^{1/2} = \log |t|$ is a $\BMO$-function on the real line.
\section{The non-autonomous problem on the real line}
\label{Sec: real line}

\noindent We begin by investigating the non-autonomous problem on the real line. So, following our previously introduced notation on forms inducing elliptic operators in divergence form, we assume that $A: \IR \times \Omega \to \IC^{n \times n}$ is bounded, measurable, and coercive in the sense that
\begin{align*}
\lambda |\xi|^2 \leq \Re (A(t,x)\xi \cdot \cl{\xi}) \qquad \text{and} \qquad |A(t,x)\xi \cdot \zeta| \leq \Lambda |\xi| |\zeta| 
\end{align*}
hold for all $t \in \IR$, a.e.\ $x \in \Omega$, and all $\xi, \zeta \in \IC^n$. We let $\a: \IR \times V \times V \to \IC$ be the corresponding non-autonomous form defined as in \eqref{Eq: Differential form}. Next, $\nabla_V: V \to \L^2(\Omega)^n$ denotes the gradient operator defined on $V$ and $\nabla_V^*: \L^2(\Omega)^n \to V^*$ is its adjoint. As a matter of fact, $\A(t) = \nabla_V^* A(t, \cdot) \nabla_V$. Here, and throughout, we identify $A(t,\cdot)$ with the corresponding multiplication operator on $\L^2(\Omega)^n$. 

In the following we write $\H^{1/2}(\IR; H)$ for the Hilbert space of all $u \in \L^2(\IR; H)$ with $\dhalf u \in \L^2(\IR; H)$, keeping in mind that $H = \L^2(\Omega)$. We define the `energy space'
\begin{align*}
 \E := \H^{1/2}(\IR; H) \cap \L^2(\IR; V)
\end{align*}
equipped with the Hilbertian norm
\begin{align*}
 \|u\|_\E^2 := \int_\IR \|u(t)\|_{\L^2(\Omega)}^2 + \|\dhalf u(t)\|_{\L^2(\Omega)}^2 + \|\nabla_V u(t)\|_{\L^2(\Omega)^n}^2 \d t.
\end{align*}
From Plancherel's theorem we obtain $\H^1(\IR; V^*) \cap \L^2(\IR; V) \subseteq \E$ with continuous embedding. By density of this embedding (note that $\H^1(\IR; V) \subseteq \E$) the bounded parabolic operator 
\begin{align*}
 \fL: \H^1(\IR; V^*) \cap \L^2(\IR; V) \to \L^2(\IR; V^*), \quad \fL u := u' + \A u
\end{align*}
naturally extends its action to a bounded operator $\E \to \E^*$, also denoted by $\fL$, via
\begin{align*}
 \fL(u)(w) = \int_\IR - (\dhalf u \,\big| \, \dhalf \HT w)_H + \langle \A u, w \rangle \d t \qquad (u, w \in \E).
\end{align*}
Here, $\HT$ denotes the Hilbert transform defined on $\L^2(\IR; H)$ by the Fourier symbol $\i \sgn(\tau)$. Besides other things, the next lemma shows that the part of $\fL$ in $\L^2(\IR; H)$ is maximal accretive. We consider this an easy though fundamental observation in the field of non-autonomous parabolic problems. It implicitly appeared in \cite{Dier-Zacher, Kaj} without being mentioned.

\begin{lem}
\label{Lem: WP on R}
Let $\theta \in \IC$ with $\Re \theta > 0$. The following assertions hold.
\begin{enumerate}
 \item For each $f \in \E^*$ there exists a unique $u \in \E$ such that $(\theta + \fL) u = f$. Moreover,
  \begin{align*}
  \|u\|_\E \leq \sqrt{2} \max \Big\{\frac{\Lambda +1}{\lambda}, \frac{|\Im \theta| + 1}{\Re \theta} \Big\} \|f\|_{\E^*}.
  \end{align*}
  
  \item If in addition $f \in \L^2(\IR; V^*)$, then $u \in \H^1(\IR; V^*)$ and $u$ is the unique solution of the non-autonomous problem
  \begin{align*}
  u'(t) + \theta u(t) + \A(t)u(t) = f(t) \qquad (t \in \IR)
  \end{align*}
  in the class $\H^1(\IR; V^*) \cap \L^2(\IR; V)$.
  
  \item If even $f \in \L^2(\IR; H)$, then
  \begin{align*}
  \|u\|_{\L^2(\IR; H)} \leq \frac{1}{\Re \theta}\|f\|_{\L^2(\IR; H)}.
  \end{align*}
  In particular, the part of $\fL$ in $\L^2(\IR; H)$ is maximal accretive with domain 
  \begin{align*}
   \mathsf{D} = \{u \in \H^1(\IR; V^*) \cap \L^2(\IR; V): \fL u \in \L^2(\IR; H)\}.
  \end{align*}
\end{enumerate}
\end{lem}

\begin{proof}
The proof following \cite{Kaj, Dier-Zacher} relies on some hidden coercivity of the parabolic operator $\fL$. It can be revealed using the Hilbert transform $\HT$. We define the sesquilinear form $\fe: \E \times \E \to \IC$ by
\begin{align*}
 \fe(v,w) = \int_\IR - (\dhalf v \mid \dhalf \HT(1+\delta \HT)w )_H + \langle (\theta + \A) v, (1+\delta \HT)w \rangle \d t,
\end{align*}
with $\delta > 0$ still to be chosen. Clearly $\fe$ is bounded. Since $\HT$ is skew-adjoint,
\begin{align*}
 \Re \int_\IR  (f \mid \HT f)_H \d t = 0 \qquad (f \in \L^2(\IR; H)).
\end{align*}
Using this along with the ellipticity of $\a$, we find for all $v \in \E$ that
\begin{align*}
 \Re \fe(v,v) 
\geq \int_\IR \delta \|\dhalf v\|_{\L^2(\Omega)}^2 + (\lambda - \delta \Lambda)\|\nabla v\|_{\L^2(\Omega)^n}^2 + (\Re \theta - \delta |\Im \theta|) \|v\|_{\L^2(\Omega)}^2 \d t.
\end{align*}
Choosing $\delta$ such that the factors in front of the second and third term are no less than $\delta$, we obtain the coercivity estimate
\begin{align*}
 \Re \fe(v,v) \geq \min \Big \{\frac{\lambda}{\Lambda + 1}, \frac{\Re \theta}{|\Im \theta|+1} \Big \} \|v\|_\E^2 \qquad (v \in \E).
\end{align*}
The Lax-Milgram lemma yields for each $f \in \E^*$ a unique $u \in \E$ with bound as required in (i) such that
\begin{align*}
 \fe(u,w) = f((1+\delta \HT)w) \qquad (w \in \E). 
\end{align*}
Since $\delta < 1$, Plancherel's theorem yields that $1+\delta \HT$ is an isomorphism on $\E$. Thus,
\begin{align*}
 \int_\IR - (\dhalf u \mid \dhalf \HT w)_H + \langle (\theta + \A) u, w \rangle \d t = f(w) \qquad (w \in \E),
\end{align*}
that is, $(\theta + \fL) u  = f$. This completes the proof of (i). If in addition $f \in \L^2(\IR; V^*)$, then the previous identity with $w \in \H^1(\IR; V)$ rewrites
\begin{align*}
 \int_\IR - \langle u, w'\rangle + \langle (\theta + \A) u, w \rangle \d t = \int_\IR \langle f, w \rangle \; \d t,
\end{align*}
thereby proving $u \in \H^1(\IR; V^*)$. As $\H^1(\IR; V)$ is dense in $\L^2(\IR; V)$, we have $u' + \theta u + \A u = f$ as a pointwise equality in $\L^2(\IR; V^*)$. Since $\H^1(\IR; V^*) \cap \L^2(\IR; V) \subseteq \E$, the first part of the proof gives of course uniqueness of $u$ in the smaller space. This proves (ii).

Finally, we prove (iii). From the first two items we can infer that $\theta + \fL: \mathsf{D} \to \L^2(\IR; H)$ is one-to-one. In order to check the resolvent estimate required for maximal accretivity, let $f \in \L^2(\IR; H)$. Then, by accretivity of $A$ and skew-adjointness of the Hilbert transform
\begin{align*}
 \Re \theta \|u\|_{\L^2(\IR; H)}^2 
 &\leq \Re \int_\IR - (\dhalf u \mid \HT \dhalf u)_H + \langle (\theta + \A)u, u \rangle \d t \\
 &= \Re \int_\IR (f \mid u)_H \d t \leq \|f\|_{\L^2(\IR; H)} \|u\|_{\L^2(\IR; H)}. \qedhere
\end{align*}
\end{proof}

Next, we provide an easy `improved time-regularity result' under the additional assumption that $A$ satisfies the Lipschitz condition
\begin{align*}
 |A(t,x) - A(s,x)| \leq C |t-s| \qquad (\text{a.e.\ $x \in \Omega$, $s, t \in \IR$})
\end{align*}
for some $C>0$. Later we will apply this result \emph{qualitatively} to smoothened versions of $A$ in order to justify some of our calculations.

\begin{lem}
\label{Lem: Higher regularity}
Let $f \in \H^1(\IR; V^*)$ and suppose that $A$ satisfies the Lipschitz condition. Then the solution $u = (1+\fL)^{-1}f$ provided by Lemma~\ref{Lem: WP on R} belongs to $\H^1(\IR; V)$.
\end{lem}

\begin{proof}
The proof is a straightforward application of the method of difference quotients. For $h \in \IR$ define
\begin{align*}
 D_h u(t) = \frac{1}{h}(u(t+h) - u(t)) \qquad (t \in \IR)
\end{align*}
and similarly define $D_h f$ and $D_h \A$. Since $u' + u + \A u = f$ on the real line, subtracting the equations for $u(t+h)$ and $u(t)$ shows that $D_h u$ is the unique solution of
\begin{align*}
 (D_h u)'(t) + D_hu(t) + \A(t) D_h u(t) = D_h f(t) - D_h \A (t) u(t+h)
\end{align*}
in $\H^1(\IR; V^*) \cap \L^2(\IR; V)$. The right-hand side can be bounded in $\L^2(\IR; V^*)$ uniformly in $h$ as we obtain 
\begin{align*}
 \int_\IR \|D_h f(t)\|_{V^*}^2 \d t \leq \int_\IR \|f'(t)\|_{V^*}^2 \d t
\end{align*}
and 
\begin{align*}
 \int_\IR \|D_h \A (t) u(t+h) \|_{V^*}^2 \d t \leq C^2 \int_\IR \|u(t)\|_V^2 \d t
\end{align*}
from our assumptions. Since $1+\fL: \E \to \E^*$ is an isomorphism, see Lemma~\ref{Lem: WP on R}, the norm $\|D_h u\|_{\L^2(\IR; V)}$ can be bounded uniformly in $h$ as well. Hence, $\{D_h u\}_{h \in \IR}$ has a weak limit point $v \in \L^2(\IR; V)$ which straightforwardly reveals itself as the weak derivative of $u$.
\end{proof}

We continue by quoting a commutator estimate due to Murray \cite[Thm.~3.3]{Murray}. The reader may also see it as a consequence of the famous $T(1)$-theorem of David and Journ\'{e} \cite{T1}.

\begin{prop}
\label{Prop: Murray commutator}
Let $a: \IR \to \IR$ be bounded and Lipschitz continuous. Then the commutator $[a, \dhalf] = a \dhalf - \dhalf a$ extends from $\H^{1/2}(\IR)$ to a bounded operator on $\L^2(\IR)$ if and only if $\dhalf a \in \BMO(\IR)$. In this case its norm is controlled by $\|\dhalf a\|_{\BMO(\IR)}$. 
\end{prop}

Since $\L^2(\IR; H) = \L^2(\Omega; \L^2(\IR))$ by Fubini's theorem, we obtain a similar result on $\L^2(\IR; H)$ for free. This is the part of the argument that crucially uses that $H$ is not an arbitrary Hilbert space.

\begin{cor}
\label{Cor: multidimensional commutator}
Suppose that $A$ satisfies the Lipschitz condition and that there exists a constant $M^\natural$ such that $\|\dhalf A(\cdot, x)\|_{\BMO(\IR)^{n \times n}} \leq M^\natural$ for a.e.\ $x \in \Omega$. Then the commutator estimate
\begin{align*}
 \|[A, \dhalf] u\|_{\L^2(\IR; H)^n} \leq C M^\natural \|u\|_{\L^2(\IR; H)^n} \qquad (u \in \H^{1/2}(\IR; H)^n)
\end{align*}
holds true with $C$ depending only on the dimension $n$.
\end{cor}

\begin{proof}
If $u \in \H^{1/2}(\IR; H)^n$, then $u(\cdot,x) \in \H^{1/2}(\IR)^n$ for a.e.\ $x \in \Omega$. Hence, the commutator $[A(\cdot, x), \dhalf] u(\cdot, x)$ is \emph{a priori} defined and the claim follows on applying Murray's estimate coordinate-wise and integrating with respect to $x$.
\end{proof}

We are in a position to prove our main result on maximal regularity of the non-autonomous problem on the real line.

\begin{prop}
\label{Prop: key}
If there is a constant $M^\natural$ such that $\|\dhalf A(\cdot, x)\|_{\BMO(\IR)^{n \times n}} \leq M^\natural$ for a.e.\ $x \in \Omega$, then for every $f \in \L^2(\IR; H)$ the solution $u = (1+\fL)^{-1}f$ of the non-autonomous problem 
\begin{align*}
 u'(t) + u(t) + \A(t)u(t) = f(t) \qquad (t \in \IR)
\end{align*}
belongs to $\H^1(\IR; H)$. More precisely, for some $C>0$ depending on $M^\natural$, $\lambda$, $\Lambda$, and $n$,
\begin{align*}
 \|u\|_{\H^1(\IR; H)} + \|u\|_{\H^{1/2}(\IR; V)} \leq C \|f\|_{\L^2(\IR; H)}.
\end{align*}
\end{prop}

\begin{proof}
Let us first establish the required estimate under the additional qualitative assumptions that $f \in \H^1(\IR; H)$ and $A$ satisfies the Lipschitz condition. In this case Lemma~\ref{Lem: Higher regularity} guarantees the higher regularity $u \in \H^1(\IR; V)$. Hence $\dhalf u \in \H^{1/2}(\IR; V) \subseteq \E$ and $\dhalf f \in \H^{1/2}(\IR; H) \subseteq \E^*$, which justifies the calculation
\begin{align}
\label{Eq1: key}
 \dhalf u = (1+\fL)^{-1}(\fL \dhalf - \dhalf \fL) u + (1+\fL)^{-1} \dhalf f
\end{align}
as an equality in $\E$. For the first term we use $\fL = \partial_t + \nabla_V^* A \nabla_V$, so that having canceled the commutating terms,
\begin{align}
\label{Eq2: key}
 (\fL \dhalf - \dhalf \fL) u = \nabla_V^*(\dhalf A - A \dhalf) \nabla_V u.
\end{align}
Since $\nabla_V u \in \H^1(\IR; H)^n$, we deduce from Corollary~\ref{Cor: multidimensional commutator} the bound
\begin{align*}
 \|(\fL \dhalf - \dhalf \fL) u \|_{\L^2(\IR; V^*)} 
 &\leq \|(\dhalf A - A \dhalf) \nabla_V u\|_{\L^2(\IR; H)^n} \\
 &\leq CM^\natural \|\nabla_V u\|_{\L^2(\IR; H)^n}.
\end{align*}
Going back to \eqref{Eq1: key}, we now obtain from Lemma~\ref{Lem: WP on R} that
\begin{align*}
 \|u\|_\E + \|\dhalf u\|_\E
 & \leq \sqrt{2} \cdot \frac{\Lambda +1}{\lambda} \Big(\|f\|_{\E^*} + \|(\fL \dhalf - \dhalf \fL) u\|_{\E^*} + \|\dhalf f\|_{\E^*} \Big) \\
 &\leq  \sqrt{2} \cdot \frac{\Lambda +1}{\lambda} \Big(2 \|f\|_{\L^2(\IR; H)} + CM^\natural \|u\|_{\L^2(\IR; V)} \Big).
\intertext{Since $u = (1+\fL)^{-1}f$, the same lemma yields the required estimate}
 &\leq  2\sqrt{2} \cdot \frac{\Lambda +1}{\lambda}\|f\|_{\L^2(\IR; H)} + 2CM^\natural \bigg(\frac{\Lambda +1}{\lambda}\bigg)^2 \|f\|_{\L^2(\IR; H)}.
\end{align*}

Now, we treat the general case by approximation. To this end let $\rho: \IR \to [0,\infty)$ be smooth with compact support such that $\int_\IR \rho(t) \d t = 1$ and let $\rho_n(t) = n \rho(nt)$. Set $f_n := \rho_n \ast_t f$ and $A_n := \rho_n \ast_t A$, the symbol $\ast_t$ indicating the convolution on $\IR$. Clearly $f_n \in \H^1(\IR; H)$ and $A_n$ satisfies the Lipschitz condition. Also
\begin{align*}
 A_n(t,x) \xi \cdot \cl{\zeta}  = \int_\IR \rho_n(t-s) A(s,x) \xi \cdot \cl{\zeta} \d s
\end{align*}
for all $\xi, \zeta \in \IC^n$, where by construction $\rho_n$ is a positive functions with integral equal to $1$. Thus, these coefficients are bounded and coercive with the same parameters $\Lambda$ and $\lambda$ as is $A$. Similarly,
\begin{align*}
 \|\dhalf A_n(\cdot, x) \|_{\BMO(\IR)^{n \times n}} = \|\rho_n \ast_t \dhalf A(\cdot, x)\|_{\BMO(\IR)^{n \times n}} \leq M^\natural
\end{align*}
for a.e.\ $x \in \Omega$, using the translation invariance of the $\BMO$-norm and the assumption on $A$ in the second step. So, if $u_n$ is the solution of 
\begin{align*}
 u_n'(t) + u_n(t) + \A_n(t)u_n(t) = f_n(t) \qquad (t \in \IR),
\end{align*}
then from the first part of the proof we can infer that $\{u_n\}_n$ is a bounded sequence in $\H^1(\IR; H) \cap \H^{1/2}(\IR; V)$. Upon passing to a subsequence, we can assume that it has a weak limit $u \in \H^1(\IR; H) \cap \H^{1/2}(\IR; V)$. Given $v \in \L^2(\IR; V)$, we thus have
\begin{align*}
 \int_\IR (u_n' + u_n - f_n \mid v)_H  \d t 
 = -\int_\IR  \langle \A_n u_n, v \rangle \d t
 = -\int_\IR  (\nabla_V u_n \mid A_n^* \nabla_V v)_H \d t
\end{align*}
and in order to reveal $u$ as the sought-after solution we have to pass to the limit $n \to \infty$. 

This is easy for the left-hand side since $u_n' \to u'$ and $u_n \to u$ weakly and $f_n \to f$ strongly, all taking place in $\L^2(\IR; H)$. For the right-hand side we recall from Section~\ref{Subsec: Comparison} that $A$, when regarded as a function in $t$, is uniformly continuous with values in $\L^\infty(\Omega; \IC^{n \times n})$. Hence, $A^*_n(t,x) \to A^*(t,x)$ for a.e.\ $(t,x) \in \IR \times \Omega$. Eventually, the dominated convergence theorem tells us $A_n^* \nabla v \to \nabla v$ in $\L^2(\IR; H)$ and thus we can pass to the limit on the right-hand side, too.
\end{proof}

\begin{rem}
Recall that an equivalent formulation of the assumption in Proposition~\ref{Prop: key}, more in the spirit of Theorem~\ref{Thm: Main}, has been discussed in Section~\ref{Subsec: Comparison}.
\end{rem}

\begin{rem} A similar analysis can be performed for the homogeneous equation $u' + \A u = f$ on the real line, using a homogeneous version $\dot \E$ of the energy space, so that $\fL: \dot \E \to \dot \E^*$ becomes bounded and invertible. Under the assumption of Proposition \ref{Prop: key} we formally obtain that for $f\in \L^2(\IR; H) \cap \dot \E^*$ the solution $u=\fL^{-1} f$ has higher regularity
\begin{align*}
\|\partial_{t}u\|_{\L^2(\IR; H)} + \|\nabla_{V}\dhalf u\|_{\L^2(\IR; H)} \leq C (\|f\|_{\L^2(\IR; H)} + \|f\|_{{\dot \E}^*}).
\end{align*}
The care of making this observation rigorous is left to the interested reader. It will not be needed in the following.
\end{rem}
\section{The proof of the main result}
\label{Sec: proof}

\noindent In this section we give the proof of Theorem~\ref{Thm: Main} by reduction to the non-autonomous problem on the real line. This will mainly rely on the following extension lemma to the effect that an elliptic coefficient function satisfying the assumption of Theorem~\ref{Thm: Main} can be extended to a function on $\IR \times \Omega$ with half $t$-derivative in $\BMO$.

\begin{lem}
\label{Lem: extension} 
Suppose that $A: [0,T] \times \Omega \to \IC^{n \times n}$ satisfies \eqref{Eq: ellipticity} and \eqref{Eq: Ass A}. Then $A$ can be extended to a map $A^\natural: \IR \times \Omega \to \IC^{n \times n}$ in such a way that the ellipticity bounds \eqref{Eq: ellipticity} remain true for all $t \in \IR$ and 
\begin{align*}
\|\dhalf A^\natural(\cdot, x)\|_{\BMO(\IR)^{n \times n}} \leq M^\natural \qquad (\text{a.e.\ $x \in \Omega$})
\end{align*}
for some constant $M^\natural$ depending on $M$, $T$, $\Lambda$, $\lambda$, and $n$.
\end{lem}

\begin{proof}
Due to Strichartz' characterization of $\BMO$-Sobolev spaces \cite[Thm.~3.3]{Strichartz-BMOSobolev}, it suffices to construct an extension $A^\natural: \IR \times \Omega \to \IC^{n \times n}$ with ellipticity bounds as in \eqref{Eq: ellipticity} and
\begin{align}
\label{Eq: extension}
 \int_I \int_I \frac{|A^\natural (t,x) - A^\natural (s,x)|^2}{|t-s|^{2}} \d s \d t \leq M^\natural \ell(I) \qquad (\text{a.e.\ $x \in \Omega$})
\end{align}
for every interval $I \subseteq \IR$, see also Section~\ref{Subsec: Comparison}. Since the extension will be in $t$-direction only, we abbreviate $A(t,x)$ simply by $A(t)$. The proof is in two steps.

\smallskip

\emph{Step 1: Extension to $[-T, 2T]$ by even reflection.}

\smallskip

\noindent Let us define $A^\flat(t) = A(t)$ for $t \in [0,T]$ and $A^\flat(t) = A(-t)$ for $t \in [-T,0]$. Clearly, this extension is elliptic with the same parameters as $A$. We claim that \eqref{Eq: extension} holds for any interval $I \subseteq [-T,T]$. In fact, by assumption on $A$ we only have to treat the case $I = [-a,b]$, where $0 \leq a,b \leq T$. Here, we split
\begin{align*}
  \int_I \int_I \frac{|A^\flat (t) - A^\flat (s)|^2}{|t-s|^{2}} \d s \d t
  &\leq  \bigg(\int_0^a \int_0^a + \int_0^b \int_0^b \bigg) \frac{|A(t,x) - A(s,x)|^2}{|t-s|^{2}} \d s \d t \\
  &\phantom{\leq} + 2 \int_0^a \int_0^b \frac{|A (-t,x) - A (s,x)|^2}{|t+s|^{2}} \d s \d t,
\end{align*}
where the first two integrals give a contribution of at most $M(a+b) = M \ell(I)$ and in the third integral we use $|t+s| \geq |t-s|$ to get a bound by $2M \max\{a,b\} \leq 2 M \ell(I)$. Hence, $A^\flat$ satisfies the estimate with constant $3 M$. By the same procedure, we can further extend to $[-T,2T]$ by the expense of a constant $9 M$ in the estimate.

\smallskip

\emph{Step 2: Extension to the real line.}

\smallskip

\noindent We extend $A^\flat$ by zero outside of $[-T,2T]$. Then, we let $\varphi: \IR \to [0,1]$ be equal to $1$ on $[0,T]$, zero outside of $[-T/2, 3T/2]$, and connect continuously and linearly in between. The extension of $A$ that we consider is $A^\natural := \varphi A^\flat + (1-\varphi)\lambda$.

Since $\varphi$ is independent of $x$, we easily see that $A^\natural$ satisfies the same ellipticity bounds as $A$. Concerning the estimate, we first bound the left-hand side of \eqref{Eq: extension} by
\begin{align*}
 \int_I \int_I \varphi(t)^2 \frac{|A^\flat (t) - A^\flat (s)|^2}{|t-s|^{2}} \d s \d t 
 + \int_I \int_I (|A^\flat(s)|^2 + \lambda^2) \frac{|\varphi (t) - \varphi (s)|^2}{|t-s|^{2}} \d s \d t.
\end{align*}
From the support properties of $\varphi$ and $A^\flat$ along with the bound $|A^\flat(t)| \leq \Lambda$ for a.e.\ $t \in \IR$, we see that the first of these two integrals is no larger than
\begin{align*}
 \int_{I \cap \left[-\frac{T}{2}, \frac{3T}{2} \right]} & \int_{I \cap \left[\vphantom{\frac{T}{2}}-T, 2T \right]} \frac{|A^\flat (t) - A^\flat (s)|^2}{|t-s|^{2}} \d s \d t
 + \int_{I \cap \left[-\frac{T}{2}, \frac{3T}{2} \right]} \int_{I \setminus \left[\vphantom{\frac{T}{2}}-T, 2T \right]} \frac{4 \Lambda^2}{|t-s|^{2}} \d s \d t.
\end{align*}
Thanks to the outcome of Step~1, this sums up to at most $(9M + 8 \Lambda^2 T^{-1}) \ell(I)$. Similarly, for the second integral we obtain the upper bound $6(\Lambda^2 + \lambda^2)T^{-1} \ell(I)$, using that $\varphi$ is bounded and Lipschitz continuous with constant $2T^{-1}$.
\end{proof}

Having all this at hand, the proof of Theorem~\ref{Thm: Main} is rather routine.

\begin{proof}[Proof of Theorem~\ref{Thm: Main}]
In the light of Lions' abstract result on maximal regularity in $V^*$ only the existence of a solution $u \in \H^1(0,T; H) \cap \H^{1/2}(0,T; V)$ is a concern. So, let $f \in \L^2(0,T; H)$ be given and let $E_0 f \in \L^2(\IR; H)$ be its extension by zero. Also, we extend $A$ to the whole real line using Lemma~\ref{Lem: extension}, where for convenience this extension is also denote by $A$. According to Proposition~\ref{Prop: key} there exists a solution $u \in \H^1(\IR; H) \cap \H^{1/2}(\IR; V)$ to the problem
\begin{align*}
 u'(t) + u(t) + \A(t)u(t) = \e^t E_0f(t) \qquad (t \in \IR).
\end{align*}
Once we have checked $u(0) = 0$, the restriction of $\e^{-t} u(t)$ to $[0,T]$ will be the solution we are looking for and the precise estimate stated in the theorem is a consequence of Proposition~\ref{Prop: key} and Lemma~\ref{Lem: extension}.

Since $u \in \H^1(\IR; V^*) \cap \L^2(\IR; V)$, the function $\|u\|_H^2$ is absolutely continuous with derivative $\frac{\mathrm{d}}{\mathrm{d} t} \|u\|_H^2 = 2 \Re \langle u', u \rangle$. (Check this for smoothened versions of $u$ first). Thus,
\begin{align*}
 \min\{1,\lambda\} \int_{-\infty}^0 \|u\|_V^2 \d t
\leq \Re \int_{-\infty}^0 \langle u+\A u, u \rangle \d t
= - \Re \int_{-\infty}^0 \langle u', u \rangle \d t
= -\frac{1}{2} \|u(0)\|_H^2,
\end{align*}
where we have used the equation for $u$ along with $E_0 f = 0$ on $(-\infty,0)$ in the second step. Thus, $\|u(0)\|_H = 0$.
\end{proof}

\begin{rem}
\label{Rem: New proof Lions}
A combination of Lemma~\ref{Lem: WP on R} and the argument performed above gives a new, neat proof of Lions' abstract maximal regularity result in $V^*$ itself, assuming only boundedness, quasi-coercivity, and measurability of $\a$. In particular, $H$ does not have to be separable. This was done by Dier and Zacher, see \cite[Thm.~6.1]{Dier-Zacher}. 
\end{rem}
\section{Further remarks and open questions}

\noindent Let us revisit the proof of Theorem~\ref{Thm: Main}. Given $f \in \L^2(\IR; H)$, Lemma~\ref{Lem: WP on R} gave us a a solution $u$ to the non-autonomous problem $u' + u + \A u = f$ on the real line that additionally satisfied $\dhalf u \in \L^2(\IR; H)$. Then, it was the boundedness of the commutator $[A, \dhalf]$ on $\L^2(\IR; H)$ that followed from our assumption on $\A$ and in turn presented us with one-half derivative more in $\L^2(\IR; H)$, that is, maximal regularity. With this at hand, the proof could be completed by rather standard arguments. 

Analogous commutator bounds hold for any fractional derivative $D_t^\alpha$, $\alpha \in (0,1)$, see again \cite[Thm.~3.3]{Murray} or use the $T(1)$-theorem. The following result can therefore be obtained by a literal repetition of the argument. Results of the same spirit also appeared in \cite[Thm.~6.2]{Dier-Zacher} for general non-autonomous forms under different regularity assumptions.

\begin{thm}
\label{thm: fractional max reg}
Consider a non-autonomous form $\a$ inducing elliptic divergence-form operators with either Dirichlet, Neumann, or mixed boundary conditions on an open set $\Omega \subseteq \IR^n$ as defined in Section~\ref{Subsec: Notation}. Suppose there exist $\alpha \in (0, \frac{1}{2})$ and $M \geq 0$ such that
\begin{align*}
\sup_{I \subseteq [0,T]} \frac{1}{\ell(I)} \int_I \int_I \frac{|A(t,x) - A(s,x)|^2}{|t-s|^{1+2\alpha}} \d s \d t \leq M \qquad (\text{a.e.\ $x \in \Omega$}).
\end{align*}
Then, given $f \in \L^2(0,T; \H)$, the unique solution $u \in \H^1(0,T; V^*) \cap \L^2(0,T; V)$ of problem \eqref{Eq: Cauchy} satisfies
\begin{align*}
 \|u\|_{\H^{\alpha + 1/2}(0,T; H)} + \|u\|_{\H^{\alpha}(0,T; V)} \leq C \|f\|_{\L^2(0,T; H)},
\end{align*}
where $C$ depends on $\lambda$, $\Lambda$, $\alpha$, $M$, $T$, and $n$.
\end{thm}

Finally let us recall that Murray's commutator estimate from Proposition~\ref{Prop: Murray commutator} is sharp. In particular, it does not remain true if the multiplier is merely $\frac{1}{2}$-H\"older continuous \cite{Murray}. Also the proof of Proposition~\ref{Prop: key} supplies an exact factorization of the half time-derivative of the solution using the commutator $[A, \dhalf]$, compare with equations \eqref{Eq1: key} and \eqref{Eq2: key}. Guided by this, we make the following

\begin{conj}
The $\frac{1}{2}$-H\"older continuity of $\A: [0,T] \to \Lop(V, V^*)$, that is,
\begin{align*}
 \|\A(t)-\A(s)\|_{V \to V^*} \leq C|t-s|^{1/2} \qquad (t,s \in [0,T]),
\end{align*}
does not imply maximal regularity of $\a$ in $H$ in general, not even if $\a$ induces elliptic differential operators in divergence form with Dirichlet, Neumann, or mixed boundary conditions on some open set $\Omega$.
\end{conj}
\def\cprime{$'$} \def\cprime{$'$} \def\cprime{$'$}


\begin{thebibliography}{10}
\providecommand{\url}[1]{{\tt #1}}
\providecommand{\urlprefix}{URL}
\providecommand{\eprint}[2][]{\url{#2}}


\bibitem{Dautray-Lions}
\textsc{R.~Dautray} and \textsc{J.L.~Lions}. Mathematical {A}nalysis and {N}umerical {M}ethods for {S}cience and {T}echnology -- {V}olume 5 {E}volution {P}roblems {I}. Springer-Verlag, Berlin, 1992.

\bibitem{T1}
\textsc{G.~David} and \textsc{J.L.~Journ\'{e}}. \emph{A boundedness criterion for generalized Calder\'{o}n-Zygmund operators.} Ann. of Math. (2) \textbf{120} (1984), no.~2, 371--397.

\bibitem{Dier-Zacher}
\textsc{D.~Dier} and \textsc{R.~Zacher}. \emph{Non-autonomous maximal regularity in Hilbert spaces}. Available at \url{http://arxiv.org/abs/1601.05213}.

\bibitem{Fackler-QE}
\textsc{S.~Fackler}. \emph{J.L.~Lions' problem concerning maximal regularity of equations governed by non-autonomous forms}. Available at \url{http://arxiv.org/abs/1601.08012}.

\bibitem{Haak-Ouhabaz}
\textsc{B.H.~Haak} and \textsc{E.M.~Ouhabaz}. \emph{Maximal regularity for non-autonomous evolution equations}. Math. Ann. \textbf{363} (2015), no.~3, 1117--1145.

\bibitem{Lions-Problem}
\textsc{J.L.~Lions}. \'{E}quations diff\'{e}rentielles op\'{e}rationnelles et probl\`{e}mes aux limites. Die Grundlehren der mathematischen Wissenschaften, vol.~111, Springer-Verlag, Berlin-G\"ottingen-Heidelberg, 1961.

\bibitem{Murray}
\textsc{M.A.M.~Murray}. \emph{Commutators with fractional differentiation and {BMO} {S}obolev spaces}. Indiana Univ. Math. J. \textbf{34} (1985), no.~1, 205--215.

\bibitem{Kaj}
\textsc{K.~Nystr\"om}. \emph{Square functions estimates and the Kato problem for second order parabolic operators in $\mathbb R^{n+1}$}. Advances in Mathematics 293 (2016), 1-36.

\bibitem{Ouhabaz-Spina}
\textsc{E.M.~Ouhabaz} and \textsc{C.~Spina}. \emph{Maximal regularity for non-autonomous {S}chr\"odinger type equations}. J. Differential Equations \textbf{248} (2010), no.~7, 1668--1683.

\bibitem{DeSimon}
\textsc{L.~ de Simon}. \emph{Un'applicazione della teoria degli integrali singolari allo studio delle equazioni differenziali lineari astratte del primo ordine}. Rend. Sem. Mat. Univ. Padova \textbf{34} (1964), 205--223.

\bibitem{Stein}
\textsc{E.M.~Stein}. Singular {I}ntegrals and {Differentiability} {P}roperties of {Functions}. Princeton Univ. Press, Princeton, 1970.

\bibitem{Strichartz-BMOSobolev}
\textsc{R.S.~Strichartz}. \emph{Bounded mean oscillation and {S}obolev spaces}. Indiana Univ. Math. J. \textbf{29} (1980), no.~4, 539--558.

\bibitem{Hendrik}
\textsc{H.~Vogt}. \emph{Equivalence of Pointwise and Global Ellipticity Estimates}. Math. Nachr. \textbf{237} (2002), no.~1, 125--128.
\end{thebibliography}
\end{document}